\documentclass [12pt]{article}

\usepackage{amssymb,amsmath,amsthm}
\usepackage{enumerate}
\usepackage{xcolor}
\usepackage{graphicx}

\newtheorem{thm}{Theorem}[section]
\newtheorem{lemma}[thm]{Lemma}

\newtheorem{cor}[thm]{Corollary}
\newcounter{nalg} 
\newenvironment{alg}{\vspace{2mm} {\bf Algorithm \arabic{nalg}.}}{\addtocounter{nalg}{1} \vspace{2mm}}

\theoremstyle{remark}

\theoremstyle{remark}

\theoremstyle{definition}

\newtheorem{ex}[thm]{Example}

\numberwithin{equation}{section}

\newcommand{\R}{\mathbb{R}}

\usepackage{a4wide} 
\newcommand{\vn}{\mathbf{v}} 
\newcommand{\wn}{\mathbf{w}} 
\newcommand{\en}{\mathbf{e}}
\newcommand{\bn}{\mathbf{b}}
\newcommand{\zn}{\mathbf{z}}
\newcommand{\xn}{\mathbf{x}} 
\newcommand{\yn}{\mathbf{y}} 
\newcommand{\nn}{\mathbf{n}} 
\newcommand{\uno}{\mathbf{1}} 
\DeclareMathOperator{\im}{im} 
\newcommand{\pe}[2]{\left\langle #1,#2 \right\rangle}
\DeclareMathOperator{\tr}{Tr}
\DeclareMathOperator{\spa}{span}

\begin{document}

\title{Computationally efficient orthogonalization for pairwise comparisons method}
\author{
Julio Ben\'itez
\thanks{jbenitez@mat.upv.es. 
Instituto de Matem\'atica Multidisciplinar, Universitat Polit\`ecnica de Val\`encia, Valencia, Spain} 
\and 
Waldemar W. Koczkodaj
\thanks{wkoczkodaj@cs.laurentian.ca. Computer Science, Laurentian University, Ontario, Canada}
\and
Adam Kowalczyk
\thanks{kowa@unimelb.edu.au. Department of Computing and Information Systems, The
	University of Melbourne, Melbourne, VIC, Australia}
}
\date{2024-01-27} 
\maketitle

\begin{abstract}
Orthogonalization is one of few mathematical methods conforming to mathematical
standards for approximation. Finding a consistent PC matrix of a given an inconsistent PC matrix is the main goal of a pairwise comparisons method. We introduce an orthogonalization for pairwise comparisons matrix based on a generalized Frobenius inner matrix product. The proposed theory is supported by numerous examples and visualizations.

$ \ $
		

\end{abstract}

\section{Introduction}

Assessing the best alternative is vital for practically every decision-making process. For it, a pairwise comparisons method is often used. Its introduction is attributed to R. Llull in 13th century (see, \cite{Llull}). 
By comparing the relative importance of criteria in pairs, 
a pairwise comparisons (PC or PCs depending on the context) matrix $M$ is created. Its entries, $m_{ij}$, are ratios (computed or assessed) of the compared entities $E_i$ and $E_j$. Entities could be physical (e.g., areas of shapes) or abstract (e.g., software quality) concepts. The accuracy of PC method was analyzed in a Monte Carlo study (see \cite{K1996}).
 
In the final step of the pairwise comparisons method, a vector of priorities is computed. Coordinates of this vector are the weights of alternatives. The largest coordinate corresponds to the best alternative. For practical use, it should be normalized to a total of one. It reflects the percentage contribution of individual coordinates.

Each entry $m_{ij} > 0$ of a PC matrix measures the preference of alternative $i$ over alternative $j$, and if
we use multiplicative comparisons, one has $m_{ji} = 1 / m_{ij}$ (called \textit{reciprocity}). Evidently, $m_{ii}=1$ since an alternative $i$ is equal to itself. 

More often than not, the inconsistency of human experts occurs in a PC matrix for the size $n \geq 3$. This means that for a PC matrix, it is very probable that there exists a cycle of alternatives indexed by $i,j,k$ such that $m_{ij} m_{jk} \neq m_{ik}$. Therefore, inconsistency analysis is needed. 

By using a logarithm mapping applied to each PC matrix $M$, we create PC matrix $A$. The conditions $m_{ji}=1/m_{ij}$ and $m_{ij} m_{jk} = m_{ik}$ are transformed into linear conditions $a_{ji}=-a_{ij}$ and $a_{ij} + a_{jk} = a_{ik}$. It enables us to use linear algebra methods. 
One of the tools for dealing with the inconsistency of the pairwise comparison matrices is to find the closest consistent matrix under a suitable distance in the set of $n \times n$ matrices. If we define an inner product in the set of $n \times n$ PC matrices, this closest consistent matrix can be found by an orthogonal projection onto a consistent subspace of PC matrices. 
To find this projection, we find a basis of the aforementioned consistent subspace. If such a basis ${\bf v}_1, \ldots, {\bf v}_k$ 
is orthogonal, the projection of any vector ${\bf v}$ is: 
\[ 
\sum_{i=1}^k \frac{ \langle {\bf v}, {\bf v}_i \rangle } {\| {\bf v}_i \|^2 } {\bf v}_i.
\]

The first orthogonal basis for computing a consistent approximation to a PC matrix was proposed in \cite{KO1997}. Subsequently, it was improved in \cite{KSS2020}.

This presentation is organized as follows. In Section \ref{s2}, we analyze additively consistent PC matrices
of order $n$, denoted by $l_n$. Moreover, 
we generalize the standard Frobenius inner product (i.e, $\pe{A}{B} = \tr(AB^T)$) obtaining a modified 
inner product depending on an arbitrary positive definite matrix $W$ (i.e., $\pe{A}{B}_W = \tr(AWB^T)$), 
where $\tr(\cdot)$ stands for the trace of a matrix.
It generates the decomposition
$g_n = l_n \perp h_{n,W}$, where $g_n$ is the subset composed of $n \times n$ skew-symmetric matrices.

In Section \ref{s3}, we obtain a $W$-orthogonal basis of $l_n$.
In Section \ref{s4}, we get a characterization of $h_{n,W}$ which determines a PC matrix belongs to $h_{n,W}$.

In Section \ref{s5}, we use graph theory to propose a (non-orthogonal) basis of $h_n$. No computation is needed to find this kind of basis. 
	
\section{Theoretical foundations for decomposition of PC matrices} \label{s2}

Let us denote by $M_{n,n}$ the set of $n \times m$ matrices and by $M_{n,n}^+$ the subset of $M_{n,n}$ of positive matrices.
We will consider any vector of $\R^n$ as a column (i.e., a matrix in $M_{n,1}$). A matrix $A = [a_{ij}] \in M_{n,n}^+$ is called  
{\em consistent} if $a_{ij} a_{jk} = a_{ik}$ for all $1 \leq i,j,k \leq n$ and is called {\em reciprocal} if $a_{ji} = 1/a_{ij}$ for all $1 \leq i,j \leq n$. Any consistent PC matrix is reciprocal but not all reciprocal PC matrices are consistent.
If $U$ and $W$ are subspaces of the vector space $V$, we denote
by $V = U \perp W$ the situation $V = U \oplus W$ together $U=W^\perp$. 

For $A=[a_{ij}] \in M_{n,n}^+$, we can define PC matrix $\mu(A) = [\ln (a_{ij})]$. Also, given $B \in M_{n,n}$, we can define the matrix
$\varphi(B) = [\exp (b_{ij})]$. Evidently, if $A \in M_{n,n}^+$, then $A$ is reciprocal if and only if $\mu(A)$ is skew-symmetric.

If $M$ is a reciprocal PC matrix (i.e., $m_{ji}=1/m_{ij}$ for all  
$i,j$), then $A=\mu(M)$ satisfies $a_{ji}=-a_{ij}$ for all $i,j$.
And, If $M$ is a consistent PC matrix (i.e., $m_{ij} m_{jk} = m_{ik}$  
for all $i,j,k$), then $A=\mu(M)$ satisfies $a_{ij} + a_{jk} = a_{ik}$ for all $i,j,k$.

By using a logarithm mapping applied to each mulitiplicative PC matrix  $M$, we can create an additive PC matrix $A$.
If a matrix $M$ is consistent, then the elements of $A=\mu(M)$ satisfy
\begin{equation*}
a_{ik}+a_{kj}=a_{ij}
\end{equation*}
for every $i,j,k$ $\in \{1, \ldots, n\}$. Such matrix $A$ is named as {\em additively consistent}.
	
Let us denote by $U_n$ the $n \times n$ matrix having all its components equal to 1.
For every dimension \(n > 0\), the following group
\begin{equation*}
G_n := \{ A \in M_{n,n}^+  : A \cdot A^T=U_n \}
\end{equation*}
is an abelian group of \(n \times n\) matrices with the operation
\begin{equation*}
\begin{array}{c}
\cdot : G_n \times G_n \rightarrow G_n \\
(A,B) \rightarrow A\cdot B =[a_{ij}\cdot b_{ij}]
\end{array}
\end{equation*}
where ``\(\cdot\)'' is the Hadamard product. Furthermore, \(G_n\) is a Lie group; its tangent space $T_{U_n}(G_n)$ of 
\(G_n\) at unity \(U_n\) consists of all \(n\times n\) skew-symmetric matrices (admitting a natural structure of a linear space and also of a Lie algebra):
\begin{equation*}
g_n:=\lbrace{\mu(A):A \in G_n}\rbrace=\lbrace{B \in M_{n,n}:\ b_{ij}+b_{ji}=0}\rbrace.
\end{equation*}

The reader is encouraged to consult \cite{KMY} for a more comprehensive treatment of Lie groups and Lie algebras of PC matrices.
The dimension of $g_n$ is $\frac{1}{2}(n^2-n)$.\\
	
\noindent For the subspace $l_n$ composed of additively consistent PC matrices:
\[
l_n:=\{B\in g_n:\ b_{ij}+b_{jk}+b_{ki}=0\}
\]
The following fundamental equivalences hold:
\begin{equation}\label{ln}
\begin{split} 
B \in l_n & \iff  \varphi(B) \text{ is consistent} \\
& \iff \text{ there exist } w_1, \ldots, w_n > 0 \text{ such that } [\varphi(B)]_{ij} = w_i/w_j \\
& \iff \text{ there exist } v_1, \ldots, v_n \in \R \text{ such that } [B]_{ij} = v_i - v_j.
\end{split}
\end{equation}
Obviously, every additively consistent matrix $A=[ a_{ij}]$ is skew symmetric, i.e., $a_{ij} = -a_{ji}$ for all $i,j$.

Let us use the linear mapping $f_n : \R^n \to M_{n,n}$ defined by:
\begin{equation}\label{define_fn}
f_n(\vn) = \vn \uno_n^T - \uno_n \vn^T,
\end{equation}
where $\uno_n = [1, \ldots , 1]^T \in \R^n$. From definition we have $[f(\vn)]_{ij} = v_i - v_j$. 

The last of the fundamental equivalences implies that $l_n$ is the image subspace of the mapping $f_n$. Evidently, the kernel of $f_n$ is the subspace spanned by $\uno_n$. Consequently, $\dim l_n = \dim \im f_n = \dim \R^n - \dim \ker f_n = n - 1$.

The basic properties of the trace operator are as follows:
$\tr:M_{n,n} \to \R$.
\begin{lemma}
\begin{enumerate}
[{\rm a)}]
\item The trace is a linear operator, i.e., $\tr(A+B) = \tr(A)+\tr(B)$ and $\tr(\alpha A) = \alpha \tr(A)$
for any $A,B \in M_{n,n}$ and $\alpha \in \R$.
\item $\tr(A) = \tr(A^T)$ for any $A \in M_{n,n}$.
\item If $A \in M_{n,m}$ and $B \in B_{m,n}$, then $\tr(AB)=\tr(BA)$.
\end{enumerate}
\end{lemma}

The third property will be frequently used for $m=1$ since $BA$ is a scalar.

The characterization of $l_n$ given by the existence of a (ranking) vector as in (\ref{ln}) is the most important fact for many applications. The existence of such an (exact) vector is rather exceptional however, and the best one may hope for is an approximal existence of it.

In order to formalize this, let us consider a linear subspace of matrices $\mathcal{L}$ such that
\[ 
l_n \subset \mathcal{L} \subset M_{n,n}
\]
and a quadratic norm
\[ 
\| \cdot \|^2: \mathcal{L} \to \R^n, \qquad A \mapsto \| A \|^2 := \langle A | A \rangle, 
\]
where $\langle \cdot | \cdot \rangle$ is is an inner product on $\mathcal{L}$. Given a matrix $M \in \mathcal{L}$, 
we wish to find its closest consistent approximation as follows:
\begin{equation}\label{minimal}
M^* = \{ A \in l_n: \| A - M \|^2 \text{ is minimal} \}.
\end{equation}
It is well known (see any standard book on linear algebra) that $M^*$ is equal to the unique matrix in $l_n$, 
being the $\langle \cdot | \cdot \rangle$-orthogonal projection of $M$ onto $l_n$.

Let us consider the special case where $\| \cdot \|$ is the Frobenius norm, i.e., 
\[ 
\langle A | B \rangle = \tr(AB^T) = \sum_{i,j=1}^n a_{ij} b_{ij},
\]
for each $A = (a_{ij}), B=(b_{ij}) \in \mathcal{L}$. In such a case the solution of the optimization (\ref{minimal}) 
can be given in a closed form by the row averages of $M$ (see \cite{simple}): 
\[ 
M^* = \frac{1}{n} f_n(M \uno_n).
\]
Equivalently, using the coordinate representation for $M = [m_{ij}]$, we have
\[ 
m_{ij}^* = \frac{1}{n} \sum_{k=1}^n (m_{jk}-m_{ik}), \quad \forall 1 \leq i,j \leq n.
\]
However, for a generic inner product on $\mathcal{L}$, the direct computation of $M^*$ could be very challenging, 
especially, if dimensionality $n$ becomes high. This is the case of some applications to genomics where $n$ is measured 
in multiple thousands. The following fundamental result shows a path for constructive solution to such a challenge.

\begin{thm}
Given a linear $\mathcal{L}$ subspace of matrices containing consistent matrices, an inner product 
$\langle \cdot | \cdot \rangle$ on $\mathcal{L}$ and a matrix $M \in \mathcal{L}$.
\begin{enumerate}[\rm (i)]
\item There exists an inner product $\langle \cdot | \cdot \rangle_n$ on $\R^n$ such that for any set of 
$\langle \cdot | \cdot \rangle_n$-orthonormal vectors $\en_1, \en_2, \ldots, \en_{n-1} \in \R^n$ which are 
$\langle \cdot | \cdot \rangle_n$-orthogonal to vector $\uno_n \in \R^n$, the matrices
\begin{equation}\label{bes}
B_1 = f_n(\en_1), \ldots, B_{n-1} = f_n(\en_{n-1}) 
\end{equation}
form an $\langle \cdot | \cdot \rangle$-orthonormal basis for the linear subspace $l_n$ of $\mathcal{L}$.
\item The solution of the optimization (\ref{minimal}) for the closest consistent approximation has the form
\[ 
M^* = \sum_{i=1}^{n-1} \langle M | B_i \rangle B_i.
\]
\end{enumerate}
\end{thm}
\begin{proof}
The proof of (i) is given by construction. As we mentioned before, the kernel of $f_n$ is the subspace
spanned by $\uno_n$. We define the mapping $\langle \cdot | \cdot \rangle_n: \R^n \times \R^n \to \R$ as
\begin{equation}\label{defineip}
\langle \xn | \yn \rangle_n = (\xn^T \uno_n) (\yn^T \uno_n) + \langle f_n(\xn) | f_n(\yn) \rangle.
\end{equation}
It is obvious that this mapping is symmetric and bilinear:
\[ 
\langle \xn | \yn \rangle_n = \langle \yn | \xn \rangle_n, \qquad 
\langle c\xn | \yn \rangle_n = c\langle \xn | \yn \rangle_n, \qquad 
\langle \xn +\zn | \yn \rangle_n = \langle \xn | \yn \rangle_n\langle \zn | \yn \rangle_n.
\]
for any vectors $\xn, \yn, \zn \in \R^n$ and any constant $c \in \R$. It is also positive definite, since for any 
vector ${\bf 0} \neq \xn \in \R^n$ we have
\[ 
\langle \xn | \xn \rangle_n = (\xn^T \uno_n)^2 + \langle f_n(\xn) | f_n(\xn) \rangle > 0,
\]
since both terms on the right-hand-side of the equation above are always nonnegative and either:
\begin{enumerate}[(a)]
\item $f_n(\xn) \neq 0$ in which case $\langle f(\xn) | f(\xn) \rangle > 0$ since $\langle \cdot | \cdot \rangle$ is
an inner product on $\mathcal{L}$, or
\item $\xn = c \uno_n \in \ker (f_n)$, $c \neq 0$ and $(\xn^T \uno_n)^2 = n^2 c^2 > 0$.
\end{enumerate}
We conclude that $\langle \cdot | \cdot \rangle_n$ is an inner product on $\R^n$. In such a case orthonormal vectors 
$\en_i \in \R^n$, $i=1,2, \ldots, n-1$, as in (\ref{bes}) always exist, and for any such selection, by the 
definition of the mapping $\langle \cdot | \cdot \rangle$ given in (\ref{defineip}), we have
\[ 
\langle f_n(\en_i) | f_n(\en_j) \rangle = \langle \en_i | \en_j \rangle_n = 
\begin{cases}
1 & \text{if } i=j, \\ 0 & \text{otherwise.} 
\end{cases}
\]
for any $1 \leq i,j \leq j \leq n-1$.

This completes the proof of part (i) of the theorem. The part (ii) follows from (i) using a well-known argument 
used in linear algebra or optimization theory (Fourier expansion in a finite dimensional vector space).
\end{proof}

If $W$ is a symmetric positive definite $n \times n$ matrix,
then ${\bf u}^T W {\bf v}$ defines an inner product for ${\bf u}, {\bf v} \in \R^n$
(for example, see \cite[Example 7.3]{poole}). The aforementioned observations extend to modification of the standard Frobenius inner product in 
the space of $n \times n$ matrices as follows: $\pe{A}{B} = \tr(AB^T)$.

We will use the following inner product in $M_{n,n}$: 
\[
\pe{A}{B}_W = \tr(AWB^T),
\]
where $W$ is a positive definite matrix. For $W=I_n$ ($I_n$ is the identity matrix of order $n$),
this inner product reduces to the standard Frobenius inner product, and for the sake of the brevity, we will
denote $\pe{\cdot}{\cdot} = \pe{\cdot}{\cdot}_{I_n}$. $\pe{\cdot}{\cdot}_W$ is indeed an inner product. 

The use of this matrix $W$ is justified because some comparisons may be important than others. 
This kind of weighting can be done by a variance-covariance matrix $W$. This matrix is definite positive and the 
meaning of the entry $w_{ij}$ is the covariance between entities $X_i$ and $X_j$. Observe that 
the covariance between $X_i$ and $X_i$ is the variance of $X_i$. The largest $w_{ii}$ is, the more important
is the assessment $i$ is. If $w_{ij}>0$ ($w_{ij} < 0$), then the assessment $i$ and $j$ are positive (negative) 
correlated: 
the more important the assessment $i$ is, the more (less) important the assessment $j$ is. 

We will consider several inner products in $M_{n,n}$ calling matrices $A$ and $B$ $W$-orthogonal for $\pe{A}{B}_W=0$, and  simply call $A$ and $B$ orthogonal for $\pe{A}{B}=0$.

We will denote $h_{n,W}$ the orthogonal complement in $g_n$ of $l_n$ by using the aforementioned inner product
$\pe{\cdot}{\cdot}_W$. When $W=I_n$, we will simply denote $h_n$. Observe that $g_n=l_n\oplus h_{n,W}$.
A study of some algebraic properties of the subspaces $g_n$, $l_n$, and $h_n$ was done in \cite{fbc}.

Let us define: 
\[
L_n:=\varphi(l_n) \qquad \text{and} \qquad H_{n,W}:=\varphi(h_{n,W})
\]
getting $G_n \simeq L_n \times H_{n,W}$. 

In \cite[Theorem 6.4]{KMY}, this decomposition is done for $W=I$ and $n=3$.
Consequently, we have:
\[ 
g_n = l_n \perp h_{n,W}.
\]
For $W=I_n$, the subspace $h_n$ was analyzed by Barzilai in \cite{Barzilai}, where he called this kind of matrices
as totally inconsistent.
	
Recalling that the orthogonal complement is taken in $g_n$, the dimension of $h_{n,W} = l_n^\bot$ is: 
\[ 
\dim g_n - \dim l_n = \frac{n^2- n}{2} - (n-1) = \frac{(n-1)(n-2)}{2}.
\]

We summarize the fundamental dimensions: 
\[ 
\dim g_n = \frac{n(n-1)}{2}, \qquad \dim l_n = n-1, \qquad \dim h_{n,W} = \frac{(n-1)(n-2)}{2},
\]


\section{$W$-orthogonal basis of $l_n$} 
\label{s3} 

In this section, we will 
find a $W$-orthogonal basis of $l_n$. 
It will allow us projecting (orthogonally) any matrix of 
$g_n$ onto $h_{n,W}$ and $l_n$. By using the the elementwise exponential mapping $\varphi$, we can decompose an arbitrary reciprocal PC matrix.

We can establish a bijection from the set of $n \times n$ skew-symmetric matrices to $\R^{n(n-1)/2}$ in the following way:
\begin{equation}\label{isomorphism}
B = \left[ \begin{array}{ccccc} 
0 & b_{12} & b_{13} & \cdots & b_{1n} \\ 
-b_{12} & 0 & b_{23} & \cdots & b_{2n} \\
-b_{13} & -b_{23} & 0 & \cdots & b_{3n} \\ 
\vdots & \vdots & \vdots & \ddots & \vdots \\
-b_{1n} & -b_{2n} & -b_{3n} & \cdots & 0
\end{array} \right]  
\cong
[b_{12}, b_{13}, \ldots, b_{n-1,n}]^T.
\end{equation}
The standard Frobenius inner product of skew-symmetric $n \times n$ matrices corresponds to the standard inner product in $\R^{(n^2-n)/2}$ as we show by the next theorem:

Let $B = [b_{ij}]$ and $C = [c_{ij}]$ be matrices. It is well known that the Frobenius inner product satisfies
$\pe{B}{C} = \tr(B C^T) = \sum_{ij} b_{ij} c_{ij}$. So, if in addition $B$ and $C$ are skew-symmetric, then 
$\pe{B}{C} = 2 \sum_{i<j} b_{ij} c_{ij}$.

Therefore, the concepts of orthogonality and the orthogonal projection are the same if we
consider the inner product $\pe{B}{C} = \tr(BC^T)$ or via the isomorphism given in \eqref{isomorphism}.

Subsequently, we will find a $W$-orthogonal basis of $l_n$. To this end, we need the following lemma
(the special case $W=I_n$ was given in \cite{rsme}).
\begin{lemma}
Let the mapping $f_n:\R^n \to M_{n,n}$ be defined as in \eqref{define_fn}. Let
$\vn, \wn \in \R^n$ and let $W$ be a positive definite matrix in $M_{n,n}$. Then 
\[ 
\pe{f_n(\vn)}{f_n(\wn)}_W = 
(\uno_n^T W \uno_n) \wn^T \vn - (\uno_n^T W \wn) (\uno_n^T \vn) - (\vn^T W \uno_n) (\wn^T \uno_n) + n (\vn^T W \wn).
\] 
For $W=I_n$, the above expression is reduced to:
\[
\pe{f_n(\vn)}{f_n(\wn)}_W = 2n \vn^T \wn - 2(\vn^T \uno_n)(\wn_n^T \uno_n).
\]
\end{lemma}
\begin{proof}
Recall that $\xn^T W \yn$ is a scalar for arbitrary $\xn, \yn \in \R^n$. We have
\begin{equation*}
\begin{split}
\langle f_n(\vn)&, f_n(\wn) \rangle_W = \tr(f_n(\vn) W f_n(\wn)^T) \\
& = \tr \left[ (\vn \uno_n^T - \uno_n \vn^T) W (\wn \uno_n^T - \uno_n \wn^T)^T\right] \\
& = \tr \left[ \vn \uno_n^T W \uno_n \wn^T - \vn \uno_n^T W \wn \uno_n^T - \uno_n \vn^T W \uno_n \wn^T + 
     \uno_n \vn^T W \wn \uno_n^T \right] \\
& = (\uno_n^T W \uno_n) \tr(\vn \wn^T) - (\uno_n^T W \wn) \tr(\vn \uno_n^T) - \\ 
& \phantom{= (} - (\vn^T W \uno_n) \tr(\uno_n \wn^T) + (\vn^T W \wn) \tr(\uno_n \uno_n^T) \\
& = (\uno_n^T W \uno_n) \tr(\wn^T \vn) - (\uno_n^T W \wn) \tr(\uno_n^T \vn) \\ 
& \phantom{= (} - (\vn^T W \uno_n) \tr(\wn^T \uno_n) + (\vn^T W \wn)\tr(\uno_n^T \uno_n) \\
& = (\uno_n^T W \uno_n) \wn^T \vn - (\uno_n^T W \wn) (\uno_n^T \vn) - (\vn^T W \uno_n) (\wn^T \uno_n) + n (\vn^T W \wn). 
\end{split}
\end{equation*}
The proof is finished.
\end{proof}
As a consequence of this formula we have that if $\vn, \wn \in \R^n$ are orthogonal to $\uno_n$ (we consider the standard
inner product in $\R^n$, i.e., $\xn^T \yn$), then
\begin{equation}\label{ex1} 
\pe{f_n(\vn)}{f_n(\wn)}_W  = (\uno_n^T W \uno_n) \wn^T \vn  + n (\vn^T W \wn) =
\vn^T \left[ (\uno_n^T W \uno_n) I_n + n W \right] \wn. 
\end{equation}

Formula~\ref{ex1} leads us to the following theorem.

\begin{thm}\label{th_julio1}
Let us define the mapping $f_n:\R^n \to M_{n,n}$ as in \eqref{define_fn}.
Let $W$ be a positive definite matrix in $M_{n,n}$.
Let $\yn_1, \ldots, \yn_{n-1} \in \R^n$ be non-zero vectors orthogonal to $\uno_n$. 
Define $M = (\uno_n^T W \uno_n) I_n + n W$.
\begin{enumerate}[{\rm a)}]
\item If $\yn_i M \yn_j = 0$ for any $i \neq j$, then 
$\{f_n(\yn_1), \ldots, f_n(\yn_{n-1})\}$ is a $W$-orthogonal basis of $l_n$.
\item If $\{\yn_1, \ldots, \yn_{n-1} \}$ is an orthogonal basis of the orthogonal complement of $\uno_n$ in $\R^n$ then
$\{f_n(\yn_1), \ldots, f_n(\yn_{n-1})\}$ is an orthogonal basis of $l_n$. 
\end{enumerate}
\end{thm}
\begin{proof}
\begin{enumerate}[{\rm a)}]
\item By \eqref{ex1}, one has $\pe{f_n(\yn_i)}{f_n(\yn_j)}_W=0$ for any $i \neq j$.
Evidently, $f_n(\vn_i) \neq {\bf 0}$ for any $i$ 
(if not, $\vn_i \in \ker f_n = \spa\{ \uno_n \}$, which contradicts $\yn_i \neq {\bf 0}$ and
$\yn_i$ is orthogonal to $\uno_n$). Since $\dim l_n = n-1$, we deduce that $\{f_n(\yn_1), \ldots, f_n(\yn_{n-1})\}$
is a $W$-orthogonal basis in $l_n$.
\item It follows from item a) by taking $W=I_n$. Observe that one has $M = 2nI_n$ and the condition 
$\yn_i M \yn_j = 0$ for any $i \neq j$ reduces to say that $\{ \yn_1, \ldots, \yn_{n-1} \}$ is an orthogonal system.
\end{enumerate}
\end{proof}

It is evident that the following $n-1$ vectors in $\R^n$
\begin{equation} \label{ex2}
\begin{array}{rcl}
\yn_1 & = & \left[ \begin{array}{ccccccc} 1, & -1, & 0, & 0, & 0, & \cdots & 0 \end{array} \right]^T, \\
\yn_2 & = & \left[ \begin{array}{ccccccc} 1, & 1, & -2, & 0, & 0, & \cdots & 0 \end{array} \right]^T, \\
\yn_3 & = & \left[ \begin{array}{ccccccc} 1, & 1, & 1, & -3, & 0, & \cdots & 0 \end{array} \right]^T, \\
& \cdots & \\
\yn_{n-1} & = & \left[ \begin{array}{cccc} 1, & \cdots & 1, & -n+1 \end{array} \right]^T
\end{array} 
\end{equation}
form an orthogonal basis of the orthogonal complement of $\uno_n$ in $\R^n$.


The following algorithm computes an orthogonal basis of $l_n$ for an arbitrary $n \geq 3$.

\begin{alg}

Input: $n$ (a natural number greater than 2)

Let $Y$ be the $n \times (n-1)$ zero matrix and let $\uno = [1,\ldots,1]^T \in \mathbb{R}^n$.

For $k=1, \ldots, n-1$

\qquad Let $Y_{1,k} = \cdots = Y_{k,k} = 1$ and $Y_{k+1,k} = -k$.

\qquad Let $\vn_k$ be the $k$-th column of $Y$.

\qquad Let $E_k = \vn_k \uno^T - \uno_n \vn^T$.

End

Output: The orthonormal basis of $l_n$ is $E_1, \ldots, E_{n-1}$.
\end{alg}


For $W \neq I_n$, observe that the matrix $M = (\uno_n^T W \uno_n)I_n + n W$ is positive definite because $W$ is positive definite. Therefore, in order to
find vectors $\yn_1, \ldots, \yn_{n-1}$ orthogonal to $\uno_n$ such that $\yn_i M \yn_i = 0$ for any
$i \neq j$, it is enough to apply the Gram-Schmidt process
(for details, see \cite{Gram-Schmidt})
to a basis of $l_n$ for the following inner product in $\R^n$: $(\vn | \wn) = \vn^T M \wn$.
Let us note that the vectors $\yn_1, \ldots, \yn_{n-1}$ given in \eqref{ex2} form a basis of $l_n$.

\section{Properties of $h_{n,W}$} \label{s4}

The following theorem characterizes the subset $h_{n,W} = l_n^\perp$ (in $g_n$).

\begin{thm}\label{characterizationhn}
Let $W$ be a positive definite matrix in $M_{n,n}$ and let $B \in M_{n,n}$ be a skew-symmetric matrix. Then 
\begin{enumerate}[{\rm a)}]
\item $B \in h_{n,W}$ if and only if $B W \uno_n = {\bf 0}$.
\item $B \in h_n$ if and only if $B \uno_n = {\bf 0}$.
\end{enumerate}
\end{thm}
\begin{proof}
First, let us prove item b). 
Let us define $\bn = B \uno_n$. We have for any $\vn \in \R^n$,
\begin{equation*}
\begin{split}
\pe{B}{f_n(\vn)} & = \tr(B f_n(\vn)^T) = \tr \left( B (\vn \uno_n^T - \uno_n \vn^T) ^T\right) = 
\tr \left( B (\uno_n \vn^T - \vn \uno_n^T) \right) \\
& = \tr(B \uno_n \vn^T) - \tr(B \vn \uno_n^T) = \tr(\bn \vn^T) - \tr(\vn \uno_n^T B) \\
& = \tr(\bn \vn^T) + \tr(\vn \uno_n^T B^T) = 
\tr(\vn^T \bn) + \tr(\vn (B \uno_n)^T) \\ 
& = \vn^T \bn + \tr(\vn \bn^T) = \vn^T \bn + \tr(\bn^T \vn) = \vn^T \bn + \bn^T \vn = 2 \vn^T \bn.
\end{split}
\end{equation*}
Therefore, we have, in view of $l_n = \im f_n$ and $h_n = l_n^\perp$
\begin{equation*}
\begin{split}
B \in h_n & \iff 0 = \pe{B}{f_n(\vn)} \text{ for all } \vn \in \R^n \\
& \iff 0 = \vn^T \bn \text{ for all } \vn \in \R^n \iff \bn = \mathbf{0} \iff 
B \uno_n = \mathbf{0}.
\end{split}
\end{equation*} 
To prove a), let us note that $\pe{X}{Y}_W = \tr(XWY^T) = \pe{XW}{Y}$.
\[ \begin{split}
B \in h_{n,W} & \iff 0 = \pe{B}{f_n(\vn)}_W \text{ for all } \vn \in \R^n \\
& \iff 0 = \pe{BW}{f_n(\vn)} \text{ for all } \vn \in \R^n  \iff BW \in h_n \iff BW \uno_n = {\mathbf{0}}.
\end{split} \]
The theorem is proved. \end{proof}

It is noteworthy that the characterization of item b) in the previous theorem was deduced differently than here by Barzilai in \cite[Theorems 4 and 5]{Barzilai}.

Note that if $B W \uno_n = \mathbf{0}$, then by transposing and using $B=-B^T$, $W=W^T$, we have 
$\uno_n^T W B = \mathbf{0}$, which is another characterization of skew-symmetric matrices in $h_{n,W}$.

It is worth noticing that if $B \in M_{n,n}$ is skew-symmetric, then we can uniquely decompose 
\begin{equation}\label{decompositionofB}
B = B_{h,W} + B_{l,W}, \qquad B_{h,W} \in h_{n,W} \text{ and } B_{l,W} \in l_n. 
\end{equation}
In fact, $B_{h,W}$ and $B_{l,W}$ are the $W$-orthogonal projections of $B$ onto $h_{n,W}$ and $l_n$, respectively. 

For $W=I_n$, simple formulas for $B_h$ and $B_l$ are provided in \cite{simple}: 
\begin{equation}\label{projectionontoln}
B_l = \frac{1}{n} \left[ (BU_n) - (BU_n)^T \right],
\end{equation}
where $U_n=\uno_n \uno_n^T$. 

Therefore, $B_h$ is $B-B_l$.

For an arbitrary $W$, finding $B_{l,W}$ can be done by using the orthogonal basis found in Theorem~\ref{th_julio1} 
and using standard tools of linear algebra for finding the orthogonal projection onto a linear subspace.
Once we have found $B_{l,W}$, it is enough to consider $B = B_{h,W} + B_{l,W}$ for finding $B_{h,W}$.

We can factorize a PC matrix $A$ using the skew-symmetric matrices $B_{l,W}$ and $B_{h,W}$ as follows:
consider $B=\mu(A)$, find $B_{h,W}$ and $B_{l,W}$ getting: 
\begin{equation*}
A = \varphi(B_h) \cdot \varphi(B_l).
\end{equation*}

All of the above, allows us to prove the following corollary.
 
\begin{cor} 
Let $W$ be a positive definite matrix in $M_{n,n}$. 
Suppose that a skew-symmetric matrix $B = [b_{ij}]$ is given and let $B_{h,W}$ and $B_{l,W}$ the orthogonal projections of 
$B$ onto $h_{n,W}$ and $l_n$, respectively.
Then the followings statements hold:
\begin{enumerate}[{\rm a)}]
\item The sum of elements of any row (or column) of \(B_{h,W} W\) (or $W B_{h,W}$) is zero.
\item The sum of elements of $i^{th}$ row (or column) of \(B_{l,W} W\) (or $W B_{h,W}$) is equal to the sum of elements of \(i^{th}\) 
row (or column) of \(B W\) (or $WB$).
\item The product of elements of row (or column) of \(\varphi(B_{h,W})W\) (or $\varphi(WB_{h,W})$) is equal to 1.
\item The product of $i^{th}$ row (or column) of \(\varphi(B_{l,W})W\) (or $\varphi(W B_{l,W})$) is equal to the product of elements of \(i^{th}\) row (or column) of \(\varphi(BW)\) (or $\varphi(WB)$).
\end{enumerate}
\end{cor} 
\begin{proof}
a) It is inferred by Theorem \ref{characterizationhn}. 
			
b) It is inferred by \eqref{decompositionofB} and Theorem \ref{characterizationhn}.

c) Let us write $B_{h,W}W = [c_{ij}]$. From item i) we get $\sum_j c_{ij} = 0$ for any index $j$. Therefore, 
$1 = \prod_j \exp(c_{ij})$. 

d) Let us write $B_l = [d_{ij}]$. From item ii) we get $\sum_j d_{ij} = \sum_{j} b_{ij}$ for any $j$. The 
proof of iv) is evident by using the exponential mapping.
\end{proof}

When $W=I_n$, the former corollary is simplified in a obvious way hence there is no need for the explicit statement.

Now, we will consider the most important case: $W=I_n$. If $B = [b_{ij}] \in M_{n,n}$ is skew-symmetric PC matrix then, as is easy to see, $B \uno_n = {\bf 0}$ if and only if
\begin{equation}\label{nullspaceB} 
\sum_{i=1}^{j-1} b_{ij} = \sum_{i=j+1}^n b_{ji} \qquad \text{for all } j = 1, \ldots, n.
\end{equation}
therefore, elements in $h_n$ are a subset of $g_n$.

To find an orthonormal basis of $h_n$, it is sufficient to show that \eqref{nullspaceB} in a matrix form $P \xn = {\bf 0}$, 
where $P$ is an $n \times m$ matrix, $m=n(n-1)/2$ (since $B$ is skew-symmetric, it has $n(n-1)/2$ unknown entries), 
and $\xn \in \R^m$ is the vector corresponding to $B$ in the isomorphism \eqref{isomorphism}, and find an orthonormal
basis of the null space of $P$. We know {\it a priori} that the dimension of the null space of $M$ is $(n-1)(n-2)/2$ since $\dim h_n = (n-1)(n-2)/2$.

We will analyze the case $n=3$ in next subsection.

\subsection{Geometric interpretation of the orthogonal projection to $h_{3}$ for PC matrices}

Let 
\begin{equation}\label{matrixB}
\begin{array}{cc}
B = \begin{bmatrix} 0 & x & y \\ -x & 0 & z \\ -y & -z & 0 \end{bmatrix} \end{array}
\end{equation} 
be an arbitrary $3 \times 3$ additive PC matrix. If $B \in l_3$, then the vector $\vn = [x,y,z]^T$
satisfies the following consistency condition:
\begin{equation}\label{acc}
x-y+z=0.
\end{equation}

\noindent since the logarithmic mapping transforms $x' \cdot  z'/y'$ to equation (\ref{acc}). 

Equation (\ref{acc}) represents a plane passing through the origin being $\nn = [1,-1,1]^T$ a normal vector 
in \(\mathbb{R}^3\). 
The isomporphism given in \eqref{isomorphism} implies that the following matrix:
\begin{equation}\label{matrixN}
N = \left[ \begin{array}{ccc} 0 & 1 & -1 \\ -1 & 0 & 1 \\ 1 & -1 & 0 \end{array} \right]
\end{equation}
spans $h_3$. We also conclude that the matrix $N$ spans $h_3$ using Theorem \ref{characterizationhn} for a matrix $B$ given in
\eqref{matrixB}. For $B \uno_3 = {\bf 0}$, we get $x+y=-x+z=-y-z=0$, which leads to
$[x,y,z]^T = \alpha[1,-1,1]^T$ for some $\alpha \in \R$.

If we denote the orthogonal projections of a skew-symmetric $3 \times 3$ PC matrix $B$ to $l_3$ and $h_3$ by $B_l$ and $B_h$, respectively, then by using 
\[
h_{3} = \spa \{ N \},
\] 
where matrix $N$ is defined in \eqref{matrixN}, we have
\begin{equation} \label{tildev}
B_h = \frac{ \pe{B}{N}}{\pe{N}{N}} N .
\end{equation}
From equality (\ref{tildev}) it follows that
\begin{equation}\label{e47}
B_h = \frac{x-y+z}{3} \begin{bmatrix} 0 & 1 & -1 \\ -1 & 0 & 1 \\ 1 & -1 & 0 \end{bmatrix}
\end{equation}
and using $B=B_h+B_l$,
\begin{equation}\label{e48}
B_l = 
\frac{1}{3}\begin{bmatrix}
				0&2x+y-z&x+2y+z\\
				-2x-y+z&0&-x+y+2z\\
				-x-2y-z&x-y-2z&0
			\end{bmatrix}.  
\end{equation}

\begin{figure}[h]
	\centering
	\includegraphics[width=0.75\linewidth]{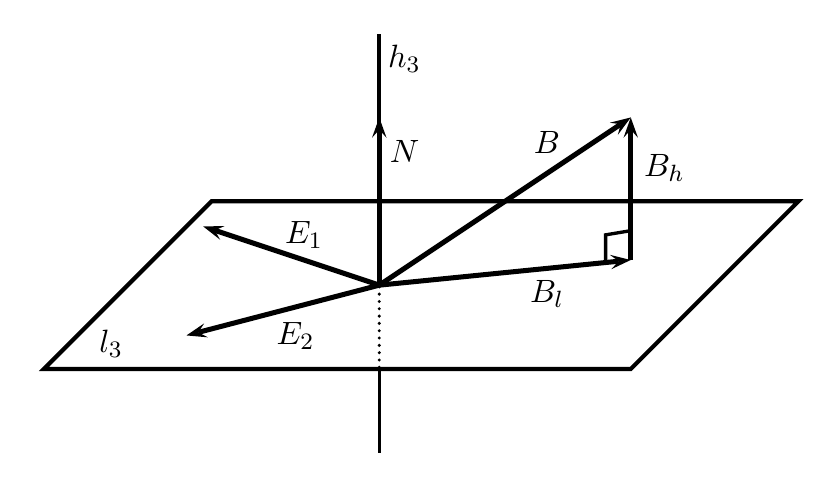}
	\caption{The tangent space  \(h_{3}\) and the normal space \(l_{3}\). The matrix $N$ spans $h_3$ and 
		the matrices $E_1$, $E_2$ span $l_3$. Any skew-symmetric matrix $B$ can be uniquely decomposed as $B = B_h+B_l$, 
		where $B_h \in h_3$ and $B_l \in l_3$. }
	\label{fig:f1}
\end{figure}

By using \eqref{projectionontoln}, we get the same expression for $B_l$.
Taking $2x+y-z=3\xi$ and $-x+y+2z=3\eta$,  we can rewrite the $B_l$ as follows:
\[ 
B_l = \xi \left[ \begin{array}{ccc} 0 & 1 & 1 \\ -1 & 0 & 0 \\ -1 & 0 & 0 \end{array} \right] + 
\eta \left[ \begin{array}{ccc} 0 & 0 & 1 \\ 0 & 0 & 1 \\ -1 & -1 & 0 \end{array} \right] =
\xi E_1 + \eta E_2.
\]
Notice that this decomposition was obtained in \cite{KMY} by using other methods. Thus, we get: 
	\begin{equation*}
		\begin{array}{cc}
			& l_{3}= \spa\lbrace{E_{1},E_{2}}\rbrace.
		\end{array}
	\end{equation*}

By using the Gram-Schmidt process in terms of Euclidean inner product $\pe{X}{Y} = \tr(XY^T)$ 
to matrices $E_1$ and $E_2$ we get an ortohogonal basis of $l_3$ formed by 
\[ 
C_1 = \left[ \begin{array}{ccc}
0 & 1 & 1 \\ -1 & 0 & 0 \\ -1 & 0 & 0
\end{array}\right], \qquad 
C_2 = \left[ \begin{array}{ccc}
0 & -1 & 1 \\ 1 & 0 & 2 \\ -1 & -2 & 0
\end{array}\right].
\]	
Hence we can apply the projection formula given at the Introduction in order to get that
the orthogonal projection of $A = [a_{ij}] \in g_3$ is
\[ 
\frac{\pe{A}{C_1}}{\| C_1 \|^2} C_1 + \frac{\pe{A}{C_2}}{\| C_2 \|^2} C_2
= \frac{2(a_{12}+a_{13})}{4} C_1  + \frac{2(-a_{12}+a_{13}+2a_{23})}{12} C_2.
\]
It would be also possible to obtain 
$\{ f_3(\yn_1), f_3(\yn_2) \}$ as an orthogonal basis of $l_3$, where
$f_3$ is defined in \eqref{define_fn} and $ \{ \yn_1, \yn_2 \}$ is an orthogonal basis of the orthogonal complement
of $\uno_3$ in $\R^3$.
 
Every PC matrix $A \in M_{3,3}$ can be factored as $A = \varphi(B_h) \cdot \varphi(B_l)$ where
$B = \mu(A)$ and matrices $B_h$, $B_l$ were previously found in an explicit way. 
	
\section{Graph theory and the elements in $h_n$} \label{s5}

Several research studies connect graph theory and pairwise comparison method. For example,
\cite{KS2015} examines the notion of principal generators of a pairwise comparisons matrix. It decreases the number of pairwise comparisons from $n(n-1)/2$ to $n-1$ for creating a consistent PC matrix. It was used in \cite{jbl2} to decrease the number of pairwise comparisons produced by an expert to get consistent results in ranking alternatives.

For a graph with $n$ vertices and $m$ oriented edges, the incidence matrix of the graph is the $n \times m$ matrix defined as follows.
The $(i,j)$-entry of this matrix is 0 if vertex $i$ and edge $j$ are not incident, and otherwise it
is 1 or $-1$ according as the edge $j$ begins or finishes at $i$; respectively. See \cite{bapat} for consulting
the definitions and basic results.

It is known that a cycle induces an element $\xn$ in the null space of an incidence matrix $P$ defined as follows. 
We define $\xn \in \R^m$ this way: 
\begin{itemize}
	\item $x_r = 0$ if the edge $r$ does not belong to a cycle,  
	\item $x_r=1$ if the edge $r$ belongs to the cycle and has the same orientation, 
	\item $x_r=-1$ if the edge $r$ belongs to the cycle and it goes in the opposite orientation.
\end{itemize} 

If we express the equations \eqref{nullspaceB} in a matrix form, $P \xn = {\bf 0}$, then the matrix $P$ is the incidence matrix of the oriented graph with $n$ nodes and the following edges: $i \to j$ if and only if $i < j$. 
We consider that the order of the edges is lexicographical.
Observe that there is $m=n(n-1)/2$ edges. 

This graph is connected hence the rank of $P$ is $n-1$. 
Therefore, the dimension of the set of solutions of $P\xn = {\bf 0}$ is 
\[m-(n-1) = \frac{n(n-1)}{2} - (n-1) = \frac{(n-1)(n-2)}{2}\]
which fully agrees with $\dim h_n$.

\subsection{A basis of $h_n$} \label{ss51}

In this subsection, we will propose a basis of $h_n$ without the need for any computation. By Theorem \ref{th_julio1}, we can decide whether an $n \times n$ matrix $B$ belongs to $h_n$, namely $B \uno_n = {\bf 0}$, or by analyzing the entries $b_{ij}$ of $B$, check whether (\ref{nullspaceB}) holds. 
As previously stated, these linear equations can be written as $P \xn = {\bf 0}$.
So, in order to find a basis of $h_n$, we will find a basis of the null space of $P$.

Subsequently, we can compute a (non orthogonal) basis of the null space of $MP$ associated to the linear system in \eqref{nullspaceB} as follows.
We consider the aforementioned graph but deleting vertex 1 and any edge beginning from the vertex 1. Observe that this
reduced graph has 
\[ \binom{n-1}{2} = \frac{(n-1)(n-2)}{2}\]
edges. If $j \to k$ is an edge of this reduced graph, then
consider the cycle $1 \to j \to k \to 1$ in the original graph and its corresponding vector belonging to the null space of $M$. In this way we get $r=(n-1)(n-2)/2$ vectors in the null space of $M$. It is sufficient to show that
these $r$ vectors are linearly independent to prove that these vectors form a basis of the null space of $M$.
From the construction of these vectors we have that the form of this vectors are: 
\[ 
\left[ \begin{array}{c} * \\ 1 \\ 0 \\ \vdots \\ 0 \\ 0 \end{array} \right], \cdots, \
\left[ \begin{array}{c} * \\ 0 \\ 0 \\ \vdots \\ 0 \\ 1 \end{array} \right] \in \R^m \]
where the blocks $[1, 0, \ldots, 0]^T, \ldots, [0,0, \ldots, 0,1]^T$ belong to $\R^r$. It is evident now that these vectors are linearly independent.

\begin{ex} \label{examplegraph}
Let us find a basis of $h_4$. Consider the graph with vertices $1,2,3,4$ and edges (note that the order of the edges is lexicographical)
\begin{equation}\label{vertices} 
1\to 2, \ 1 \to 3, \ 1 \to 4, \ 2 \to 3, \ 2 \to 4, \ 3 \to 4.
\end{equation}

The graph is presented by Fig.~\ref{fig:f2}. 

\begin{figure}[h]
\centering
\includegraphics[width=0.9\linewidth]{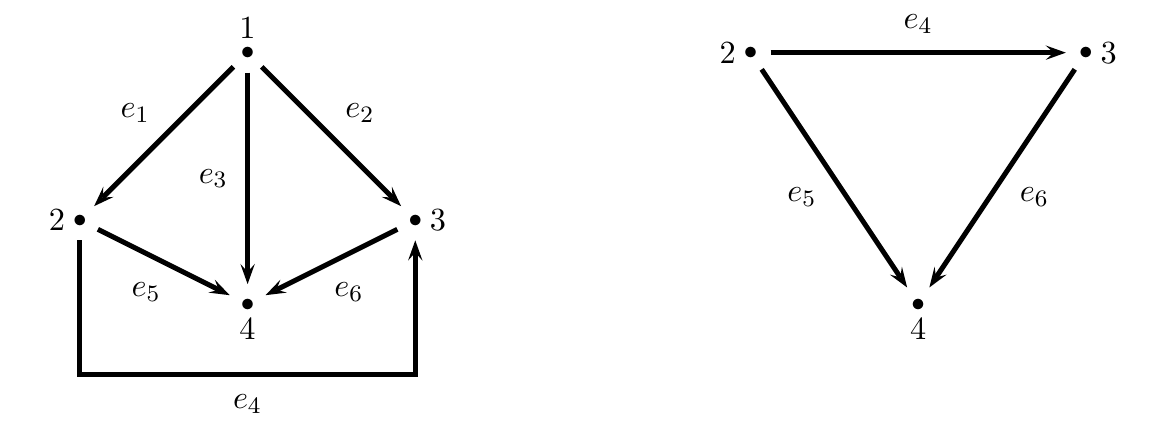}
\caption{Left: graph considered in Example \ref{examplegraph}. Right: the reduced graph is formed by deleting
node 1 and any edge containing the node 1. The order of the edges $e_1, \ldots, e_6$ is lexicographical.}
\label{fig:f2}
\end{figure}

Following the above notation, for $n=4$ and $m=n(n-1)/2 = 6$, there are $6$ edges. Let us consider the reduced graph
whose vertices are $2,3,4$ and its edges are $2 \to 3, 2 \to 4, 3 \to 4$ (we delete vertex 1 and any edge connecting
this vertex). This reduced graph has $r=3$ edges. By considering edge $2 \to 3$ we form the cycle
$1 \to 2 \to 3 \to 1$ (observe that the edge $3 \to 1$ is traveled in the opposite way), we get
vector $\xn_1 = [1, -1, 0, 1, 0, 0]^T$ because edges $1 \to 2$ and $2 \to 3$ are the first and the fourth edge in \eqref{vertices} 
(first and fourth coordinate in $\xn_1$) and they are oriented as in \eqref{vertices} and the edge $3 \to 1$ is the second edge in \eqref{vertices}, but in the opposite way (second coordinate of $\xn_1$).

In a similar way, we get the vectors $\xn_2 = [1,0,-1,0,1,0]^T$ and $\xn_3 = [0,1,-1,0,0,1]^T$. The set
$\{ \xn_1, \xn_2, \xn_3 \}$ is a basis of the solutions of $M \xn = {\bf 0}$.

If a vector $\xn = [x_i] \in \R^6$ satisfies $M \xn = {\bf 0}$, then matrix
\[ 
\left[ \begin{array}{cccc} 
0 & x_1 & x_2 & x_3 \\ -x_1 & 0 & x_4 & x_5 \\ 
-x_2 & -x_4 & 0 & x_6 \\ -x_3 & -x_5 & -x_6 & 0
\end{array} \right]
\]
belongs to $h_4$. So, we have found a basis of $h_4$:
\[ 
\left[ \begin{array}{cccc} 0 & 1 & -1 & 0 \\ -1 & 0 & 1 & 0 \\ 1 & -1 & 0 & 0 \\ 0 & 0 & 0 & 0 \end{array} \right], 
\left[ \begin{array}{cccc} 0 & 1 & 0 & -1 \\ -1 & 0 & 0 & 1 \\ 0 & 0 & 0 & 0 \\ 1 & -1 & 0 & 0 \end{array} \right],
\left[ \begin{array}{cccc} 0 & 0 & 1 & -1 \\ 0 & 0 & 0 & 0 \\ -1 & 0 & 0 & 1 \\ 1 & 0 & -1 & 0 \end{array} \right]. \]
\end{ex}

The following algorithm computes a basis of $h_n$ for an arbitrary $n>2$.

\begin{alg} 

Input: $n$ (a natural number greater than 2)

For $i=2, \ldots, n$

\qquad For $j=i+1, \ldots, n$.

\qquad \qquad Let $N$ be the $n \times n$ zero matrix.

\qquad \qquad Let $N_{1,i} = 1$, $N_{i,1} = -1$, $N_{i,j} = 1$, $N_{j,i} = -1$, $N_{1,j} = -1$, $N_{j,1} = 1$.

\qquad End for $j$

End for $i$

Output: The matrices $N$ (there are $(n-1)(n-2)/2$) computed in each double loop.

\end{alg} 

\noindent Let us notice that the assignments:

\[ N_{1,i}=1; \qquad N_{i,j}=1; \qquad N_{1,j}=-1\]
\noindent correspond to the cycle $1 \to i \to j \to 1$. Note
that the edge $j \to 1$ is in the opposite orientation as in the graph defined at the beginning of this section.
 
Let us observe that this basis is not orthogonal. But it can be orthogonalized by using the Gram-Schmidt process in a standard way.

\section{Conclusions}

A generalized Frobenius inner product was proposed for orthogonalization of a PC matrix. 
A $W$-orthogonal basis of $l_n$ can be obtained in a simpler way (see Section~\ref{s3}). 
The proposed orthogonalization creates new opportunities for applications.

A (non-orthogonal) basis of $h_n$ can be obtained (without computational effort) by using the proposed procedure specified in Subsection~\ref{ss51}. Subsequently, the classic Gram-Schmidt process can be used. Orthogonalizalion is one of few mathematical methods conforming to mathematical standards for approximation of an inconsistent PC matrix by a consistent PC matrix. However, simple heuristics may be used for complex systems requiring immediate solutions as  demonstrated in \cite{1M}.



\end{document}